\documentclass[11pt]{article}

\usepackage[letterpaper, margin=1.0in]{geometry}

\usepackage{amsmath}
\usepackage{amsfonts}
\usepackage{amsthm}
\usepackage{amssymb}
\usepackage{bm}
\usepackage{url}
\usepackage{xcolor}
\usepackage{framed}
\usepackage{comment}

\usepackage{tikz}
\usepackage{tikz-cd}
\usetikzlibrary{arrows.meta}
\usetikzlibrary{matrix,arrows,positioning,calc,shapes}

\usepackage{algorithm}
\usepackage{algorithmic}

\usepackage{graphicx}

\usepackage{verbatim}

\usepackage{dsfont}
\usepackage{mathrsfs}
\usepackage[cmtip,arrow]{xy}
\usepackage{pb-diagram,pb-xy}
\usepackage{multirow}
\usepackage{phoenician}
\usepackage{subcaption}

\theoremstyle{plain}  
\newtheorem{thm}{Theorem}[section]
\newtheorem{lem}[thm]{Lemma}
\newtheorem{prop}[thm]{Proposition}

\theoremstyle{definition}  
\newtheorem{defn}[thm]{Definition}

\newtheorem{ex}[thm]{Example}

\theoremstyle{remark}  
\newtheorem{rem}[thm]{Remark}

\newcommand{\setof}[1]{\left\{ {#1}\right\}}

\newcommand{\bzero}{{\bf 0}}
\newcommand{\bone}{{\bf 1}}

\newcommand{\R}{{\mathbb{R}}}

\newcommand{\Z}{{\mathbb{Z}}}

\newcommand{\cF}{{\mathcal F}}

\newcommand{\cI}{{\mathcal I}}

\newcommand{\cK}{{\mathcal K}}

\newcommand{\cM}{{\mathcal M}}

\newcommand{\cS}{{\mathcal S}}
\newcommand{\cT}{{\mathcal T}}

\newcommand{\cX}{{\mathcal X}}

\newcommand{\sJ}{{\mathsf J}}
\newcommand{\sK}{{\mathsf K}}
\newcommand{\sL}{{\mathsf L}}

\newcommand{\sQ}{{\mathsf Q}}

\newcommand{\mvmap}{\rightrightarrows}

\newcommand{\sInvset}{{\mathsf{ Invset}}}

\newcommand{\Con}{\mathop{\rm Con}\nolimits}

\newcommand{\sABlock}{{\mathsf{ ABlock}}}

\newcommand{\Inv}{\mathop{\mathrm{Inv}}\nolimits}

\newcommand{\Int}{\mathop{\mathrm{int}}\nolimits}

\newcommand{\id}{\mathop{\mathrm{id}}\nolimits}

\newcommand{\Prob}{\mathbb{P}}

\usepackage{authblk}

\title{Conditioned Weiner Processes as Nonlinearities: A Rigorous Probabilistic Analysis of Dynamics}

\author[a]{Konstantin Mischaikow\footnote{mischaik@math.rutgers.edu}}
\author[b]{Cameron Thieme\footnote{cameron.thieme@rutgers.edu}}
\affil[a]{Department of Mathematics, Rutgers, The State University of New Jersey, Piscataway,
NJ, 08854, USA}
\affil[b]{DIMACS, Rutgers, The State University of New Jersey, Piscataway, NJ, 08854, USA}

\begin{document}
\maketitle
\begin{abstract}
We study a Weiner process that is conditioned to pass through a finite set of points and consider the dynamics generated by iterating a sample path from this process.  Using topological techniques we are able to characterize the global dynamics and deduce the existence, structure and approximate location of invariant sets.  Most importantly, we compute the probability that this characterization is correct.  This work is probabilistic in nature and intended to provide a theoretical foundation for the statistical analysis of dynamical systems which can only be queried via finite samples.  

\end{abstract}

\section{Introduction}

Motivated by \cite{bogdan}, in this paper we identfy the dynamics generated by sample paths of a Gaussian process $f = \setof{f(x)}_{x\in X}$ conditioned to pass through a set of data points $\cT = \setof{(x_n,y_n)}_{n=0}^N$ where $X=[x_0,x_N]$ (we specify the mean and covariance of this process precisely in Section~\ref{Sec: brownian}).
More precisely, we provide a characterization of local and global dynamics on $X$ and an exact probability that this description is valid for a random sample path from $f$.  

Although for the sake of simplicity we remain in a probabilistic setting, and avoid describing any sampling procedures, we are motivated by the problem of understanding a dynamical system that we are able to access only through finite samples of the same form as $\cT$. 
To understand the significance of our results (and those of \cite{bogdan}) let us first consider a typical approach to this problem.  
Given the data set $\cT$, a variety of well-known techniques are available to generate an surrogate function $\hat{f}$ \cite{gramacy}.
If $\hat{f}$ is accepted as the model, then the standard theoretical and computational techniques of nonlinear dynamics can be applied.
However, because dynamical systems are subject to bifurcations, it is not clear that the dynamics generated by $\hat{f}$ and the true dynamics that generated the data are the same, i.e., conjugate.
In fact, the behavior of nonlinear dynamics is sufficiently rich that distinguishing conjugacy classes requires uncountably many invariants \cite{conj_prob} and thus identifying the conjugacy class is impossible with finite data.

With this in mind we propose to characterize the dynamics via the concept of a Morse tiling \cite{lattice3}.

\begin{defn}
Let $g: Y \to Y$ be a continuous function on a compact metric space $Y$ and $\sQ$ a finite partially ordered set with partial order $\preceq$.  
A {\bf Morse tiling} of $Y$ for $g$ is a decomposition of $Y$ into a collection
\[
\cM = \setof{M(q) \mid q \in \sQ}
\]
of regular closed sets with disjoint interiors with the following property.
Let  $y \in Y$.
If $g^n(y) \in M(p)$, $g^m(y) \in M(q)$, and $n<m$, then $q\preceq p$. 

The sets $M(q)$ are {\bf Morse tiles} for $g$ and the partially ordered set $\sQ$ is a {\bf Morse graph} for $g$.
\end{defn}

The partial order on $\sQ$ implies that if the trajectory leaves a Morse tile, then it cannot return.
Therefore, recurrent dynamics of $g$  occurs within individual Morse tiles.
This provides a global decomposition of the dynamics where the gradient-like behavior is characterized by the Morse graph.

As is described in Section~\ref{Sec: conley} an algebraic topological invariant, the Conley index, can be assigned to each Morse tile.
The significance of this is that the Conley index can be used  to deduce the existence of interesting invariant sets for $g$ such as fixed points, periodic orbits, heteroclinic orbits, bistability, and chaotic dynamics \cite{mischmroz}.

The main result of this paper (see Theorem~\ref{Thm:TileBnd} for the technical details) is that we can produce a Morse tiling (as well as the associated Conley indices) and provide an exact formula for the probability that this Morse tiling is valid for a random sample path from $f$. 
There is an important caveat: our Gaussian process $f$ is a sequence of Brownian bridges with variance parameter $\sigma^2$ interpolating the set $\cT$.

An outline of the paper is as follows.
In Section~\ref{Sec: brownian} we state all of our probabilistic assumptions along with a simple proposition that is necessary for our computations.  
For the convenience of the reader we recall in Section~\ref{Sec: order}  basic definitions and concepts from order theory.  
In Section~\ref{Sec: conley} we recall basic ideas from dynamical systems and combinatorial Conley theory.  
The main results of the paper are presented in Section~\ref{Sec: results}.  
Section~\ref{Sec: examples} covers several examples designed to illustrate the techniques used in this paper.  
Finally, we provide some concluding remarks and comment on future directions in Section~\ref{Sec: concl}.

\section{Brownian Paths and Excursion Bounds}\label{Sec: brownian}

In this section we state our basic probabilistic assumptions and a proposition that will be used extensively throughout the article.  We begin by recalling the definitions of some simple Gaussian processes.  

A Gaussian process is uniquely specified by its mean and covariance function \cite{adtay}.  The Gaussian process $W=\setof{W(x)}_{x \geq 0}$ such that $\text{E}(W(x)) = 0$ and $\text{Cov}(W(x),W(x')) =\min(x,x')$ is called the {\bf standard Weiner process} or {\bf standard Brownian motion}.  More generally, if $x_a,x_b,y_a,y_b \in \R$ and $x_a < x_b$ then the Gaussian process $Z = \setof{Z(x)}_{x \in [x_a,x_b]}$ with $\text{E}(Z(x)) = y_a$ and $\text{Cov}(Z(x),Z(x')) = \sigma^2 \min(x,x')$ for all $x,x' \in [x_a,x_b]$, where $\sigma >0$, 
is the {\bf Weiner process (or Brownian motion) on the interval $[x_a,x_b]$, starting at $(x_a,y_a)$, with variance parameter $\sigma^2$}.  A {\bf Brownian bridge $B = \setof{B(x)}_{x \in [x_a,x_b]}$ from $(x_a,y_a)$ to $(x_b,y_b)$ with variance parameter $\sigma^2$} has the law of $Z$ conditioned to take the value $y_b$ when $x = x_b$.   Then $B$ is a Gaussian process satisfying $\text{E}(B(x)) = y_a + \frac{x - x_a}{x_b - x_a}(y_b - y_a)$ for all $x\in [x_a,x_b]$ and $\text{Cov}(B(x),B(x')) = \sigma^2 \frac{(x - x_a)(x_b - x')}{x_b - x_a}$ for $x_a \leq x \leq x' \leq x_b$.  

With these definitions in mind we discuss the Gaussian process $f$ that we  study in this paper.  
We begin with a collection of $N+1$ points, $\cT = \setof{(x_n,y_n)}_{n=0}^N$, where we assume that $x_n < x_m$ for $n < m$; recall that we have defined $X:= [x_0,x_N]$.  
For simplicity of exposition we also assume that
$$ x_0 < y_n < x_N \text{ for } 0\leq n \leq N.$$
As indicated in the introduction, we will assume that $f$ is Brownian motion with variance parameter $\sigma^2$ conditioned on the events $\setof{f(x_n) = y_n}_{n=0}^N$.  More specifically, $f = \setof{f(x)}_{x\in X}$ is the Gaussian process with mean $\mu(x):=\text{E}(f(x))$ and covariance $\kappa(x,x'):=\text{Cov}(f(x),f(x'))$ given by
\begin{itemize}
    \item $\mu(x) = y_{n-1} + \frac{x - x_{n-1}}{x_n - x_{n-1}}(y_n - y_{n-1})$ for all $x \in  [x_{n-1},x_n]$ and 
    \item $\kappa(x,x') = 
    \begin{cases}
    \sigma^2 \frac{(x - x_{n-1})(x_n - x')}{x_n - x_{n-1}}, & x_{n-1} \leq x \leq x' \leq x_{n}\\
    0, & \text{otherwise}
    \end{cases}$.  
\end{itemize}
Note that we can think of this process as a sequence of $N$ independent Brownian bridges $B_n$ from $(x_{n-1},y_{n-1})$ to $(x_n,y_n)$ with variance parameter $\sigma^2$.  

Our ultimate goal is to characterize the dynamics of $f$ on $X$.  In order to do so we will need to know the probability that $f$ stays between two particular threshold values on various combinations of intervals.  That is, we are interested in computing the probabilities of events $$ S_n(\alpha,\beta) := \setof{f(x) \in (\alpha,\beta)\,|\, x \in [x_{n-1},x_n]}.$$
Because the process $f$ may be considered a sequence of independent Brownian bridges, this probability is equivalent to $\Prob( B_n(x) \in (\alpha, \beta) \,|\, x \in [x_{n-1},x_n])$.  We can obtain this value exactly, which is a key motivation for our choice of Gaussian process. 

In order to obtain $\Prob(S_n(\alpha,\beta))$ we will make use of the function $\pi:\R^4 \to \R$ defined by 
\[
\pi(a,b,c,d) := \sum_{m=1}^\infty 
\exp\left( -2[m^2ab + (m-1)^2cd + m(m-1)(ad + c b)]\right)
\]
\[
+ \exp\left( -2[(m-1)^2ab + m^2cd + m(m-1)(ad + c b)]\right)
\]
\[
- \exp\left( -2[m^2(ab+cd) + m(m-1)ad + m(m+1)c b]\right)
\]
\[
- \exp\left( -2[m^2(ab+cd) + m(m+1)ad + m(m-1)c b]\right).
\]

\begin{lem}[\cite{doob}, (4.3)]\label{Lem: Doob_exact}
Let $a,c \geq 0$ and $b,d > 0$.  Then
\[
\Prob\left[\sup_{0\leq x<\infty}\{W(x) - (ax + b)\} \geq 0 \,\text{OR}\, \inf_{0\leq x \leq \infty}\{ W(x) + c x +d\} \leq 0 \right]= \pi(a,b,c,d).
\]

\end{lem}

Using Lemma \ref{Lem: Doob_exact} we are able to derive the following equation.

\begin{prop}\label{prop:BasicBound}
Assume that $\alpha < \min(y_{n-1},y_n) \leq \max(y_{n-1}, y_n) < \beta$.  Then 
$$\Prob\left( S_n(\alpha,\beta) \right) 
= 1 - \pi\left(\frac{(\beta - y_n)}{\sigma \sqrt{x_n - x_{n-1}}}, \frac{(\beta - y_{n-1})}{\sigma \sqrt{x_n - x_{n-1}}} ,\frac{(y_n-\alpha)}{\sigma \sqrt{x_n - x_{n-1}}}, \frac{(y_{n-1} - \alpha)}{\sigma \sqrt{x_n - x_{n-1}}}\right).$$
\end{prop}

Notice that if either $\alpha$ or $\beta$ is in the interval $[\min(y_{n-1},y_n) , \max(y_{n-1}, y_n)]$ then the probability is trivially zero.  

\begin{proof}
In order to simplify the notation we will prove the result for a Brownian bridge $B$ with variance $\sigma^2$ from $(0,p)$ to $(T,q)$; no generality is lost with this shift.  It is straightforward to verify that $B$ may be represented as 
$$B(x) = \frac{T-x}{\sqrt{T}}\sigma W\left(\frac{x}{T-x}\right) + p + x\left(\frac{q - p}{T}\right)$$
with $B(T) = q$.  We will use the transformation $s = \frac{x}{T-x}$.  Thus we see the following:

\begin{align*}
    \sup_{x \in [0,T)} B(x) &= \sup_{x \in [0,T)} \frac{T-x}{\sqrt{T}}\sigma W(\frac{x}{T-x}) + p + x\left(\frac{q-p}{T}\right)\\
    &= \sup_{s \in [0,\infty)} \left( \frac{T - \frac{sT}{1+s}}{\sqrt{T}}\right) \sigma W(s) + p + \frac{Ts}{1+s}\left(\frac{q-p}{T}\right)\\
    &= \sup_{s \in [0,\infty)} \left( \frac{\sigma \sqrt{T}}{1+s}\right)  W(s) + p + \frac{s(q-p)}{1+s}.
\end{align*}
From this characterization we see that the following events are equivalent. 
\begin{align*}
    \sup_{x \in [0,T)} B(x) \leq \beta &\iff \sup_{s \in [0,\infty)}\left[ \left( \frac{\sigma \sqrt{T}}{1+s}\right)  W(s) + p + \frac{s(q-p)}{1+s} \right] \leq \beta \\
    &\iff \sup_{s \in [0,\infty)}\left[ \left( \frac{\sigma \sqrt{T}}{1+s}\right)  W(s) + (p - \beta) + \frac{s(q-p)}{1+s} \right] \leq 0\\
    &\iff \sup_{s \in [0,\infty)}\left[ \left( \sigma \sqrt{T}\right)  W(s) + (p - \beta)(1+s) + s(q-p) \right] \leq 0\\
    &\iff \sup_{s \in [0,\infty)}\left[ \left( \sigma \sqrt{T}\right)  W(s) -\left( s(\beta - q) + (\beta - p) \right)\right] \leq 0\\
    &\iff \sup_{s \in [0,\infty)}\left[  W(s) -\left( s\frac{(\beta - q)}{\sigma \sqrt{T}} + \frac{(\beta - p)}{\sigma \sqrt{T}} \right)\right] \leq 0\\
\end{align*}
Similary, we get that
\[
\inf_{x \in [0,T)} B(x) \geq \alpha \iff \inf_{s \in [0,\infty)}\left[   W(s) + \left( s\frac{(q-\alpha)}{\sigma \sqrt{T}} + \frac{(p - \alpha)}{\sigma \sqrt{T}} \right) \right] \geq 0.
\]

The result then follows from Lemma \ref{Lem: Doob_exact}.

\end{proof}

\section{Order Theory}\label{Sec: order}

In this section we recall some basic definitions and notations from order theory that are used in the following section on computational Conley theory.  
For a more complete introduction the reader is refered to \cite{davey:priestley}.

\begin{defn}
A {\bf lattice} $\sL$ is a set with the commutative and associative binary operations $\land,\lor: \sL \times \sL \to \sL$ which satisfy the following absorption axiom for all $a,b,c \in \sL$:
\[
a \land (a \lor b) = a = a \lor (a \land b)
\]
A lattice is {\bf distributive} if for each $a,b,c \in \sL$ it also satisfies the additional axiom
\[
a \lor (b \land c) = (a \lor b) \land (a \lor c).
\]
A lattice is {\bf bounded} if there exist {\bf neutral} elements $\bzero,\bone \in \sL$ with the property that 
\[
\bzero \land a = \bzero, \hspace{0.5cm} \bzero \lor a = a, \hspace{0.5cm} \bone \land a = a, \hspace{0.5cm} \bone \lor a = \bone 
\]
for all $a \in \sL$.
\end{defn}

All lattices used in this paper are both bounded and distributive. 

In the introduction we described our main results using the idea of a poset.  We formally record the definition of this object here.  
\begin{defn}

A {\bf partially ordered set} or {\bf poset} $\sQ $ is a set with an order relation $\leq$ satisfying the three properties:
\begin{enumerate}
    \item {\bf Reflexivity:} $a \leq a$
    \item {\bf Transitivity:} $a\leq b$ and $b \leq c$ $\implies$ $ a\leq c$
    \item {\bf Anti-symmetry:} $a \leq b$ and $b \leq a$ $\implies$ $a=b$
\end{enumerate}
\end{defn}
Given $q\in \sQ$ the {\bf downset} of $q$ is defined as $\downarrow(q):= \setof{q' \in \sQ \mid q' \leq q}$.

An element $c$ of a lattice $\sL$ is {\bf join-irreducible} if $c \neq \bzero$ and $c = a \lor b$ implies that $c = a$ or $c = b$ for all $a,b \in \sL$.  
We denote the set of join-irreducible elements of $\sL$ by $\sJ(\sL)$

Any lattice $\sL$ has a naturally induced partial order relation $\leq$; for any $a,b \in \sL$,
\[
a \leq b \iff a \land b = a.
\]
Since $\sJ(\sL)\subset \sL$ it inherits the partial order relation on $\sL$.
Furthermore, an element $c\in\sL$ is join-irreducible if and only if there exists a unique element $\overleftarrow{c} \in \sL$ such that $\overleftarrow{c} < c$ and there is no $a \in \sL$ with $a\neq \overleftarrow{c}$ such that $\overleftarrow{c} < a < c$;  the element $\overleftarrow{c} \in \sL$ is called the {\bf immediate predecessor} of $c \in \sJ(\sL)$.

Before closing this section let us specify some notation.  Given any lattice $\sL$, the notation $\sL' \hookrightarrow \sL$ will indicate that $\sL'$ is a sublattice of $\sL$; that is, $\sL'$ is a subset of $\sL$, $\sL'$ is itself a lattice and there is an inclusion morphism from $\sL'$ to $\sL$.  We will further assume that any sublattice $\sL' \hookrightarrow \sL$ introduced in this paper contains the same neutral elements $\bzero$ and $\bone$ that bounded the original lattice $\sL$. 

\section{Conley Theory}\label{Sec: conley}

Conley theory has two components: (i) a framework for global decompositions of dynamics, and (ii) algebraic topological tools for reconstructing dynamics.
For this paper the first is represented by the Morse tiling and the second by the Conley index.
For the sake of simplicity we present this theory in the setting of one-dimensional maps where the starting point is the set of data points $\cT$ (see \cite{bogdan} for a more general dimension independent discussion).

As indicated in the introduction given the set of points $\cT = \setof{(x_n,y_n)}_{n=0}^N$ we define the phase space of interest to be $X=[x_0,x_n]$.
We decompose $X$ as a simplicial complex $\cX = \cX(\cT)$ with vertices $\cX^{(0)} := \setof{x_n}_{n=0}^N$ and edges $\cX^{(1)} := \setof{[x_{n-1},x_{n}]}_{n=1}^{N}$.
Viewing the face relation as a partial order $\downarrow([x_{n-1},x_{n}])= \setof{[x_{n-1},x_{n}], x_{n-1},x_n}$. 

We use \emph{combinatorial multivalued maps} to model the dynamics. 
In particular for the dynamics computations we use $\cF^\text{top}\colon \cX^{(1)} \mvmap \cX^{(1)}$  a set valued function $\cF^\text{top}(\xi)\subset \cX^{(1)}$.
To simplify the discussion concerning the Conley index we restrict our attention to \emph{interval valued} maps, i.e.,  for each interval $\xi\in \cX^{(1)}$, $\cF^\text{top}(\xi) = \cup_{i=1}^j [x_i,x_{i+1}]$, i.e., the image of $\cF^\text{top}(\xi)$ is an interval.
In order to perform the algebraic topological computations that determine the Conley index we extend  $\cF^\text{top}$ to $\cF\colon \cX\mvmap\cX$ by setting
\begin{align*}
    \cF([x_{n-1},x_n]) &= \cF^\text{top}([x_{n-1},x_n]) \\
    \cF(x_n) &= \downarrow \left(\cF^\text{top}([x_{n-1},x_n]) \cup \cF^\text{top}([x_{n},x_{n+1}]) \right)
\end{align*}

A combinatorial multivalued map $\cF^\text{top}\colon \cX^{(1)} \mvmap \cX^{(1)}$ is an \emph{outer approximation} of  $g\colon X\to X$ if $g(x)\in \Int(\cF^\text{top}(\xi))$ for $x\in \xi$.

To tie these combinatorial multivalued maps to the information provided by $\cT$ we make use 
of the surrogate map $\mu$ and define
\begin{equation}
\label{eq:minimalMV}
    \cF_\mu^\text{top}(\xi):=\setof{\xi' \in \cX^{(1)} \mid \xi' \cap \mu(\xi) \neq \emptyset}.
\end{equation}
As is shown in \cite{KMV0} $\cF_\mu^\text{top}$ is an outer approximation of $\mu$.
Note that any \emph{enclosure} of $\cF_\mu^\text{top}$,  i.e., $\cF^\text{top}\colon \cX^{(1)} \mvmap \cX^{(1)}$ such that $\cF_\mu^\text{top}\subset \cF^\text{top}(\xi)$ for every $\xi\in\cX^{(1)}$, is an outer approximation of $\mu$.

\begin{rem}
For the remainder of this paper we restrict our attention to multivalued maps $\cF^\text{top}$ that are enclosures of $\cF_\mu^\text{top}$.
\end{rem}

To see how these combinatorial constructions relate to continuous dynamics, recall that a closed set $K\subset Y$ is an {\bf attractor block} for a continuous function $g\colon Y \to Y$ if 
\[
g(K)\subset \Int(K)
\]
where $\Int(K)$ denotes the interior of $K$.
The set of all attractor blocks for $g$ forms a bounded and distributive lattice \cite{lattice1} and is denoted by  $\sABlock(g)$.  
The bounding elements of $\sABlock(g)$ are $\bone = Y$ and $\bzero = \emptyset$.  

Given $\cF^\text{top}\colon \cX^{(1)} \mvmap \cX^{(1)}$ define $\sInvset^+(\cF^\text{top}):=\setof{\cS \subset \cX^{(1)} \mid \cF^\text{top}(\cS) \subset \cS}$; software which computes $\sInvset(\cF^\text{top})$ is available \cite{CMGDB}.
We leave it to the reader to check that if $\cF^\text{top}$ is an enclosure of $\cF_\mu^\text{top}$, then $\sInvset^+(\cF^\text{top})\hookrightarrow \sInvset^+(\cF_\mu^\text{top})$.
As is shown in \cite{lattice2}, $\sInvset^+(\cF^\text{top})$ is bounded distributive lattice with $\wedge =\cap$, $\vee = \cup$, $\bzero = \emptyset$ and $\bone = X$.
Of fundamental importance, as it leads to Proposition~\ref{prop:tiling}, is the following result of \cite{lattice2}:
if  $\cF^\text{top}$ is an outer approximation of $g\colon X\to x$, then $\sInvset^+(\cF^\text{top}) \hookrightarrow \sABlock(g)$.

\begin{prop}
\label{prop:tiling}
Given data set $\cT = \setof{(x_n,y_n)}_{n=0}^N$, simplicial complex $\cX(\cT)$, and combinatorial multivalued map $\cF^\text{top}\colon \cX^{(1)} \mvmap \cX^{(1)}$ that is an enclosure of $\cF_\mu^\text{top}$, consider a lattice $\sK$ such that $\sK \hookrightarrow \sInvset^+(\cF^\text{top}) \hookrightarrow \sInvset^+(\cF_\mu^\text{top}) \hookrightarrow \sABlock(\mu)$.
For each $K \in \sJ(\sK)$, define
\begin{equation}\label{eq: tile}
    M(K):= \text{cl}(K\setminus \overleftarrow{K})
\end{equation}
and further let
\begin{equation}\label{eq: tiling}
    \cM(\sK) := \setof{M(K) \mid K \in \sJ(\sK)}.
\end{equation}
If $\cF^\text{top}$ is an outer approximation of $g\colon X\to X$, then $\cM(\sK)$ is a Morse tiling for $g$ with Morse graph $\sJ(\sK)$.
\end{prop}

For a proof of Proposition~\ref{prop:tiling} in a more general setting see  \cite{lattice3}.
However, the intuition behind the result is simple. 
Let $x\in X$. Then $x\in M(K)$ for some $K\in\sJ(\sK)$. This implies that $x\in K\in \sABlock(g)$. Let $n>0$ and assume $g^n(x) \in M(K')$. If $K\neq K'$, then $K'\subset K$ and hence $K' \leq K$. 
 
We now turn to the Conley index and again begin our discussion on the purely combinatorial level.
We assume that we are given $\sK \hookrightarrow \sInvset^+(\cF^\text{top})$. 
We define an {\bf index pair} for $\cF^\text{top}$ to be a pair $\cK = (K_1,K_0)$ where $K_1,K_0 \in \sInvset^+(\cF^\text{top})$  and $K_0 \subset K_1$.  
Then $\cF$ induces a map on homology, $\cF*:H_*(\downarrow(K_1),\downarrow(K_0)) \to H_*(\downarrow(K_1),\downarrow(K_0))$.
We define the {\bf Conley index} of $\cK$ to be the shift equivalence class of $\cF_*$ and denote it by $\Con_*(\cK)$ \cite{mischaikow:weibel:22}.  
In particular we can identify each $K \in \sJ(\sK)$ with the index pair $(K,\overleftarrow{K})$, and so we declare the Conley index of $M(K)$ to be 
\begin{equation}\label{eq: index}
    \Con_*(M(K)):= \Con_*(\downarrow(K),\downarrow(\overleftarrow{K})).
\end{equation}

Turning to continuous dynamics consider a continuous map $g\colon X\to X$.
Given $N \subset X$, the {\bf maximal invariant set} contained in $N$ is given by 
\[
\Inv(N,g):=\setof{x \in N \mid \exists \upsilon:\Z \to N \text{ such that } \upsilon(0) = x \text{ and } \upsilon(k+1) = g(\upsilon(k)) \text{ for all } k \in \Z }.
\]
A compact set $N \subset X$ is an {\bf isolating neighborhood} if $\Inv(N,g) \subset \Int(N)$.  
If $N$ is an isolating neighborhood, then the homology Conley index $\Con_*(\Inv(N,g))$ is well defined \cite{mischmroz}. 

To tie together the combinatorial and continuous theory we note that if $\cF^\text{top}$ is an outer approximation of $g\colon X\to X$, then $K,\overleftarrow{K}\in \sABlock(g)$.
Therefore, $g\colon (K,\overleftarrow{K})\to (K,\overleftarrow{K})$ and $(K,\overleftarrow{K})$ is an index pair for the classical Conley theory \cite{mischmroz}.
Finally, $\cF*:H_*(\downarrow(K_1),\downarrow(K_0)) \to H_*(\downarrow(K_1),\downarrow(K_0))$ and $g_*\colon H_*(K,\overleftarrow{K})\to H_*(K,\overleftarrow{K})$.  Thus we obtain the following result.

\begin{prop}\label{Prop: index}
$\Con_*(M(K))$ is shift equivalent to $\Con_*(\Inv(M(K),g))$.
\end{prop}


\section{Results}\label{Sec: results}

Our construction of a Morse tiling of $X$ for a function $g:X \to X$ relied on a lattice of attractor blocks $\sK$ for $g$.  Since we are interested in obtaining a Morse tiling of $X$ for a sample path from $f$ we would therefore like to know the probability that a finite sublattice $\sK \hookrightarrow \sABlock(\mu)$ is also a lattice of attractor blocks for $f$.  

\begin{rem}
Since attractor blocks are defined for $f$ only if $f \in C(X,X)$, we are interested in the event 
\[
(f \in C(X,X)) \cap (f(K) \subset \text{Int}(K)\forall K \in \sK)
\]
that for the sake of simplicity we denote by $(\sK \hookrightarrow \sABlock(f))$.
Since $\sABlock(\mu)$ is bounded from above by $\bone = X$, and we have assumed that all sublattices will be bounded by the same elements as the original lattice, we have that $\bone = X \in \sK$.  
Thus the event $(f(K) \subset \text{Int}(K)\forall K \in \sK)$ implies the event $(f(x) \in X \forall x \in X)$.
Therefore, since $f$ is continuous with probability $1$, 
\[
\Prob(\sK \hookrightarrow \sABlock(f)) = \Prob((f(K) \subset \text{Int}(K)\forall K \in \sK)).
\]
\end{rem}

Let $\sK \hookrightarrow \sInvset^+(\cF_\mu^\text{top}) \hookrightarrow \sABlock(\mu)$; our goal is to compute $\Prob(\sK \hookrightarrow \sABlock(f))$.  We need notation to describe the elements of $\sK$ in terms of the edges in $\cX$. 
Thus, for fixed $K\in \sK$ we write $K = \bigcup_{m=1}^{M_K} J_m$ where the
\[
J_m = [x_{i_m},x_{j_m}] 
\]
are disjoint closed intervals.
Define the indexing sets $\cI_m(K) = \cI_m = \setof{n \mid i_m < n \leq j_m}$ for $1\leq m\leq M_K$.
By assumption $K$ is an attractor block for $\mu$, and therefore for  each $m$ there is some unique $\tau(m) \in \setof{1,\cdots, M_K}$ such that $\mu(J_m) \subset \Int(J_{\tau(m)})$.

Let $\cI(K) = \cup_{m=1}^{M_K} \cI_m$ and define the maps $\alpha_{K}:\cI({K}) \to \R$ and $\beta_{K}:\cI({K}) \to \R$ as follows. 
\begin{equation}
\label{eq:alphaDefn}
\text{
If $n\in \cI_{m}$ then $\alpha_{K}(n) := x_{i_{\tau(m)}}$ and $\beta_{K}(n) := x_{j_{\tau(m)}}$.
}
\end{equation}
Notice that $\alpha_{K}(n)$ (resp. $\beta_{K}(n)$) is simply the minimum (resp. maximum) of the connected component of ${K}$ which contains $y_{n-1},y_n$.  

The event that $K$ is an attracting block for $f$ is equivalent to $\bigcap_{n\in \cI({K})} S_n(\alpha_{K}(n),\beta_{K}(n))$.  Therefore we have the following result.
\begin{prop}\label{Prop:AttBlBound}
\begin{equation}
    \Prob( f({K}) \subset \Int({K})) =\prod_{n \in \cI({K})} \Prob(S_n(\alpha_{K}(n),\beta_{K}(n)))
  \end{equation}  
\end{prop}

We now aim to extend Proposition~\ref{Prop:AttBlBound} to all of $\sK$ simultaneously.
Let $\cI := \setof{1,\ldots, N}$ and define the map $\gamma:\cI \to \sK$ by
\[
\gamma(n):=\min\setof{K \in \sK \,|\, n \in \cI(K)}
\]
where the minimum is taken with respect to the partial order of $\sK$.  Note that this minimum is well-defined because if $n \in \cI(K)$ and $n \in \cI(K')$ then $K\cap K' \in \sK$ and $n \in \cI(K \cap K')$.  

Observe that the event $S_n(\alpha_{\gamma(n)}(n),\beta_{\gamma(n)}(n))$ implies the event $S_n(\alpha_{K}(n),\beta_{K}(n))$ for all $K$ such that $n \in \cI(K)$; this statement holds because $n\in \cI(K)\cap \cI(K')$ and $K < K'$ implies that $\alpha_{K'}(n) \leq \alpha_{K}(n) < \beta_{K}(n) \leq \beta_{K'}(n)$. Finally, we define the maps $\alpha:\cI \to \R$ and $\beta:\cI \to \R$ by 
\[
\alpha(n):= \alpha_{\gamma(n)}(n), \hspace{.5cm} \beta(n):= \beta_{\gamma(n)}(n).
\]
With this notation introduced we may now state one of the key results of this paper.

\begin{thm}\label{Thm:LattBound}
Let $\sK$ be a sublattice of $\sInvset^+(\cF)$.  The probability that $\sK$ is a lattice of attractor blocks for $f$ is given by the following equation. 
\begin{equation}
    \Prob( \sK \hookrightarrow \sABlock(f)) =\prod_{n \in \cI} \Prob(S_n(\alpha(n),\beta(n)))
  \end{equation}  
\end{thm}

\begin{proof}

Recall that the lattice $\sK$ is naturally imbued with a partial order $\leq$.  Choose a linear extension $\leq'$ of the partial order $\leq$.  This allows us to write $\sK = \setof{K_q \mid 1 \leq' q \leq' Q}$ where the labelling of the sets $K_q$ respects the partial order of $\sK$; that is, if $K_p \subset K_q$ then $p \leq q$.  

For each $1\leq q \leq Q$ and each $n\in \cI$, define $S_n^q$ to be the event $f([x_{n-1},x_n]) \subset (\alpha_{K_q}(n),\beta_{K_q}(n))$.  Let $V_q$ be the event that $K_q \in \sABlock(f)$; then $V_q =\cap_{n \in \cI(K_q)} S_n^q$.  Our goal is to compute
$$\Prob( \sK \hookrightarrow \sABlock(f)) = \Prob(\cap_{q=1}^Q V_q) = \Prob(V_1)\Prob(V_2|V_1)\Prob(V_3|V_1\cap V_2)\cdots\Prob(V_Q|\cap_{q=1}^{Q-1}V_q).$$
In order to compute that value we will use the following intermediary step for any $1\leq P \leq Q$.
\begin{align*}
    \Prob(V_P | \cap_{q=1}^{P-1}V_q) &= \Prob(\cap_{n\in \cI(K_p)} S_n^P | (\cap_{n \in \cI(K_1)} S_n^1)\cap(\cap_{n \in \cI(K_2)} S_n^2)\cap\cdots\cap(\cap_{n \in \cI_{P-1}} S_n^{k-1}))\\
    &= \Prob(\cap_{n\in \cI(K_P)} S_n^P | \cap_{n \in (\cup_{p = 1}^{P-1}\cI(K_p))}S_n(\alpha(n),\beta(n)))\\
    &= \Prob(\cap_{n\in \cI(K_P)\setminus (\cup_{p = 1}^{P-1}\cI(K_p))} S_n^P)\\
    &= \Prob(\cap_{n\in \cI(K_P)\setminus (\cup_{p = 1}^{P-1}\cI(K_p))} S_n(\alpha(n),\beta(n)))\\
    &= \prod_{n\in \cI(K_P)\setminus (\cup_{p = 1}^{P-1}\cI(K_p))}\Prob(S_n(\alpha(n),\beta(n)))\\
\end{align*}

The result then follows from this computation:

\begin{align*}
    \Prob(\cap_{q=1}^Q V_q)
    &= \Prob(V_1)\Prob(V_2|V_1)\Prob(V_3|V_2\cap V_1)\cdots \Prob(V_Q|\cap_{q=1}^{Q-1}V_q)\\
    &= \prod_{q = 1}^Q \prod_{n \in \cI(K_q)\setminus(\cup_{p = 1}^{q-1}\cI(K_p))} \Prob(S_n(\alpha(n),\beta(n)))\\
    &= \prod_{n \in \cI} \Prob(S_n(\alpha(n),\beta(n)))
\end{align*}

\end{proof}

As indicated in Section~\ref{Sec: conley}, knowing a lattice of attractor blocks $\sK$ for a function allows us to determine a Morse tiling of the domain for that function.  In order to compute the Conley index of each of the associated Morse tiles, however, we must have an outer approximation of the function of interest; therefore we will now construct a combinatorial multivalued map $\cF_\sK^\text{top}:\cX^{(1)}\mvmap \cX^{(1)}$ that is an outer approximation of $f$ whenever $\sK \hookrightarrow \sABlock(f)$.  

In fact, $\cF_\sK^\text{top}$ is essentially defined by the construction used in Theorem~\ref{Thm:LattBound}.  For each $\xi_n:=[x_{n-1},x_n] \in \cX^{(1)}$ let 
\begin{equation}
\label{eq:cFsK}
    \cF_\sK^\text{top}(\xi_n):=[\alpha(n),\beta(n)].
\end{equation}

\begin{prop}\label{Prop: outappx}
Let $\sK \hookrightarrow \sInvset^+(\cF_\mu)$ and define $\cF_\sK$ by Equation~\eqref{eq:cFsK}.  The following properties hold.
\begin{enumerate}
    \item $\cF_\sK^\text{top}$ is an enclosure of $\cF_\mu^\text{top}$.\label{enclosure}
    \item $\sK \hookrightarrow \sInvset^+(\cF_\sK^\text{top})$\label{invset}
    \item $\cF_\sK^\text{top}$ is an outer approximation of $f$ if and only if $(\sK \hookrightarrow \sABlock(f))$.\label{outappx}
\end{enumerate}
\end{prop}

\begin{proof}
Properties \ref{enclosure} and \ref{invset} follow directly from the construction of $\cF_\sK^\text{top}$ and Theorem~\ref{Thm:LattBound}.  Further, if $\cF_\sK^\text{top}$ is an outer approximation of $f$ then $\sInvset^+(\cF_\sK) \hookrightarrow \sABlock(f)$ \cite{lattice2} and so $\sK \hookrightarrow \sInvset^+(\cF_\sK)\hookrightarrow \sABlock(f)$.  
Therefore it remains only to show that $\sK \hookrightarrow \sABlock(f)$ implies that $\cF_\sK^\text{top}$ is an outer approximation for $f$; we prove this by the contrapositive and assume that $\cF_\sK^\text{top}$ is not an outer approximation for $f$.  Then there exists $x$ such that $f(x) \not \in \Int(\cF_\sK^\text{top}(\xi_n))$ where $x \in \xi_n = [x_{n-1},x_n]$. Consider $\gamma(n) \in \sK$ and write $\gamma(n) = \bigcup_{m=1}^{M_{\gamma(n)}} J_m$.  Let $m'$ be the integer such that $\xi_n \subset J_{m'}$.  By construction $\cF_\sK^\text{top}$ is constant on each $J_m$ and hence $f(x) \not \in \Int(\cF_\sK^\text{top}(J_{m'}))= \Int(\cF_\sK^\text{top}(\xi_n))$.  Again by construction $\cF_\sK^\text{top}(J_{m'}) = J_{\tau(m')}$ and so $f(x) \not \in \Int(J_{\tau(m')})$; thus $f(J_{m'}) \not \subset \Int(J_{\tau(m')})$ and so $f(K) \not \subset \Int(K)$ (here we use the fact that the $J_m$ are disjoint intervals, and so $f(K) \subset \Int(K)$ only if each of these intervals map into the interior of another one).
\end{proof}

We now obtain the main result of this paper as a direct consequence of Theorem~\ref{Thm:LattBound} and Propositions~\ref{prop:tiling}, \ref{Prop: index}, and \ref{Prop: outappx}.  

\begin{thm}\label{Thm:TileBnd}
Let $\sK \hookrightarrow \sInvset^+(\cF_\mu^\text{top})$ and define $M(K)$ and $\Con_*(M(K))$ for each $K \in \sJ(\sK)$ by Equations~\eqref{eq: tile} and \eqref{eq: index} respectively.  Then with probability 
\[\prod_{n \in \cI} \Prob(S_n(\alpha(n),\beta(n)))\]
$\cM(\sK) = \setof{M(K) \mid K \in \sJ(\sK)}$ is a Morse tiling of $X$ for $f$ and $\Con_*(M(K))$ is the Conley index of each Morse tile.  
\end{thm}


\section{Examples}\label{Sec: examples}

This section contains three examples that help to explain the results of Section~\ref{Sec: results}.  In each example we plot the mean $\mu$ of the Gaussian process $f$ as well as a random sample path from $f$.  The sample path plotted does not come into our analysis of the examples in any way and is merely included in order to illustrate the ideas at play.

In this first example we begin with an attractor block identified for the mean $\mu$ that is a single closed interval; a single closed interval is the simplest form of an attractor block possible.   We will apply Proposition~\ref{Prop:AttBlBound} and give the probability that this interval is an attractor block for the Gaussian process of interest.   

\begin{ex}\label{Ex: intvl}

Consider  
\begin{align*}
\cT &= \setof{(x_n,y_n) \mid x_n = n/10}_{n=0}^{10}\\
&= \{(0,.4509),(.1,.4999),(.2,.4613),(.3,.455),(.4,.5185),\\
&\, (.5,.4987),(.6,.5398),(.7,.5147),(.8,.5397),(.9,.5221),(1,.5331)\}.
\end{align*}
We will let $f$ be Brownian motion with variance parameter $\sigma^2 = 1$ conditioned to pass through $\cT$.  The mean function of $f$, $\mu$, is the piecewise linear function indicated in blue, shown in Figure~\ref{Fig:ex1_all}.  The combinatorial multivalued map $\cF_\mu^\text{top}$ is shown in Figure~\ref{Fig:ex1_F}; $K := X = [0, 1]$ is an element of $\sInvset^+(\cF_\mu^\text{top})$ and hence $K\in \sABlock(\mu)$.  More directly, we can see that $K$ is an attractor block for $\mu$ precisely because $0 < f(x_n) < 1$ for all $n \in \setof{0,\cdots,10}$, and the piecewise linearity of $\mu$ then implies that $ 0 < f(x) < 1$ for all $x \in [0,1]$.

\begin{figure}[H]

\centering

\begin{subfigure}{0.48\textwidth}
    \includegraphics[width=\textwidth]{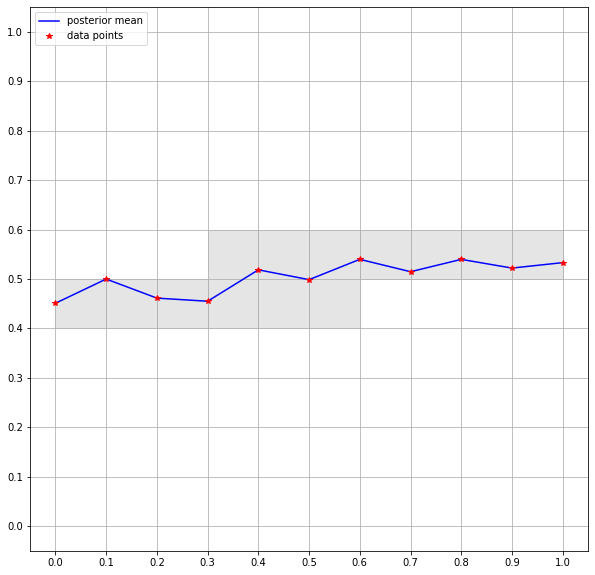}
    \caption{The combinatorial multivalued map $\cF_\mu^\text{top}$ is shown above.  Cells on the horizontal axis map to the collection of cells on the vertical axis that indicated by the grey coloring.}
    \label{Fig:ex1_F}
\end{subfigure}
\hfill 
\begin{subfigure}{0.48\textwidth}
    \includegraphics[width=\textwidth]{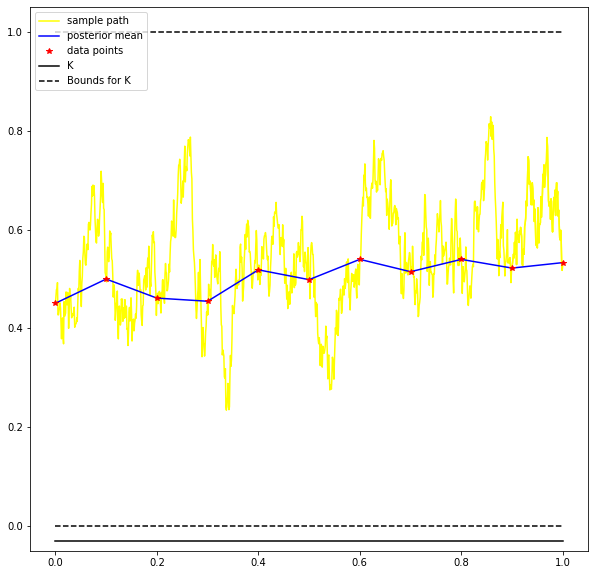}
    \caption{The interval $[0, 1]$ is an attracting block for the sample path because the sample path lies entirely between $0$ and $1$ on the interval $[0,1]$.  The dashed lines define $\cF_\sK^\text{top}$.}
    \label{Fig:ex1_main}
\end{subfigure}
\caption{}
\label{Fig:ex1_all}
\end{figure}

In this case, $\alpha_K(n)$ is identically $\alpha:= 0$, and $\beta_K(n)$ is identically $\beta:=1$. By Proposition~\ref{Prop:AttBlBound} the probability that $K \in\sABlock(f)$ is
\[
\prod_{n=1}^{10} 
     1 - \pi\left(\frac{(\beta - y_n)}{\sigma \sqrt{x_n - x_{n-1}}}, \frac{(\beta - y_{n-1})}{\sigma \sqrt{x_n - x_{n-1}}} ,\frac{(y_n-\alpha)}{\sigma \sqrt{x_n - x_{n-1}}}, \frac{(y_{n-1} - \alpha)}{\sigma \sqrt{x_n - x_{n-1}}}\right)  
\approx
0.8586.
\]

Alternatively, we can view the set $\sK = \setof{\emptyset, K=X}$ as a lattice of attractor blocks for $\mu$ and apply Theorem~\ref{Thm:LattBound} to determine that $\Prob(\sK \hookrightarrow \sABlock(f)) = 0.8586$ as well.  In this example the Morse tiling of $X$ is trivial and $\cM(\sK) = \setof{M(K)=X}$.  Also,
\[
\Con_k(M(K)) = \begin{cases}
\id & \text{if $k=0$} \\
0 & \text{otherwise}
\end{cases}
\]
which implies the existence of a fixed point \cite{szymczak:96,day:frongillo:trevino}.
\end{ex}

In this next example we consider a larger lattice of sets and use Theorem~\ref{Thm:LattBound} in order to identify the probability that this lattice is made up of attractor blocks for $f$.  While the probabilistic ideas are not really any different from the preceding example, the indexing required to keep track of this more sophisticated structure is more complicated.  We remark that this example demonstrates how the techniques developed in this paper may be used to identify bistability in a system.

\begin{ex}

We consider the set 
\begin{align*}
\cT &= \setof{(x_n,y_n) \mid x_n = n/10}_{n=0}^{10}\\
&= \{(0,.2005),(.1,.225),(.2,.2057),(.3,.2025),(.4,.2343),\\
&\, (.5,.5243),(.6,.8449),(.7,.8324),(.8,.8448),(.9,.8361),(1,.8416)\}
\end{align*}
shown in Figure~\ref{Fig:ex2_all}, and let $f$ be Brownian motion with variance parameter $\sigma^2 = 1/16$ conditioned to pass through $\cT$.  We err on the side of verbosity in analyzing this data in order to be totally clear how the indexing works when dealing with a general lattice of attracting blocks.

Using $\cF_\mu^\text{top}$, shown in Figure~\ref{Fig:ex2_F}, we identify five attracting blocks for the posterior mean $\mu$ which form a lattice $\sK$:
$$ K_0 = \emptyset, K_1 = [x_{0},x_{4}], K_2 = [x_{6},x_{10}], K_3 = K_1 \cup K_2, K_4 = [x_{1}, x_{10}]$$
Note that $K_0 \subset K_1 \subset K_3 \subset K_4$ and $K_0 \subset K_1 \subset K_2 \subset K_4$, but $K_1$ and $K_2$ are incomparable. 

The indexing sets are
\begin{align*}
    \cI(K_1) & = \setof{1,2,3,4}, \\
    \cI(K_2) & = \setof{7,8,9,10}, \\
    \cI(K_3) & = \setof{1,2,3,4,7,8,9,10}, \\ 
    \cI(K_4) &= \setof{1,2,3,\cdots,9,10}.
\end{align*} 
For $ q\neq 3$, each component $K_q$ consists only of a single connected component and thus the maps $\alpha_{K_q}:\cI(K_q) \to \R$, $\beta_{K_q}:\cI(K_q) \to \R$ are constant maps:
$$ \alpha_{K_1} \equiv x_{0}, \beta_{K_1} \equiv x_{4}$$
$$ \alpha_{K_2} \equiv x_{6}, \beta_{K_2} \equiv x_{10}$$
$$ \alpha_{K_4} \equiv x_{0}, \beta_{K_4} \equiv x_{10}$$
The maps $\alpha_{K_3}$ and $\beta_{K_3}$ are slightly more complicated since $K_3$ has two connected components, but we do not need to describe them because $\gamma(n) \neq K_3 $ for any $n \in \cI$.  We do need to know $\gamma(n)$ for each $n \in \cI$:
\[
\gamma(n) = 
\begin{cases}
K_1, & n \in \cI(K_1)\\
K_2, & n \in \cI(K_2)\\
K_4, & \text{otherwise}
\end{cases}
\]

\begin{figure}[H]

\centering

\begin{subfigure}{0.48\textwidth}
    \includegraphics[width=\textwidth]{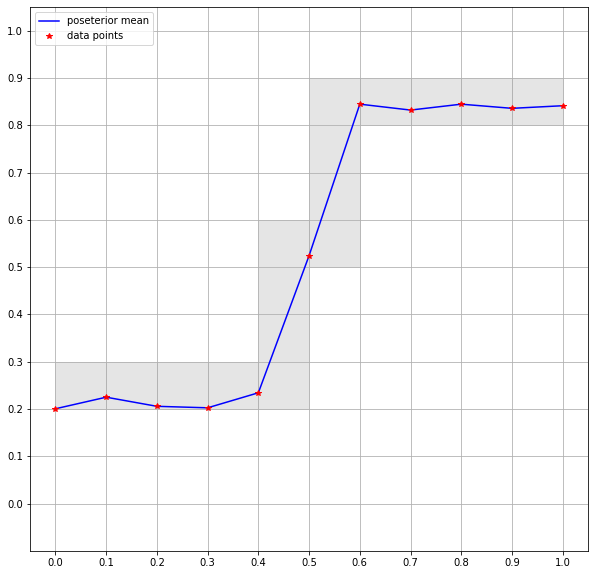}\caption{The combinatorial multivalued map $\cF_\mu^\text{top}$.}
    \label{Fig:ex2_F}
\end{subfigure}
\hfill 
\begin{subfigure}{0.48\textwidth}
    \includegraphics[width=\textwidth]{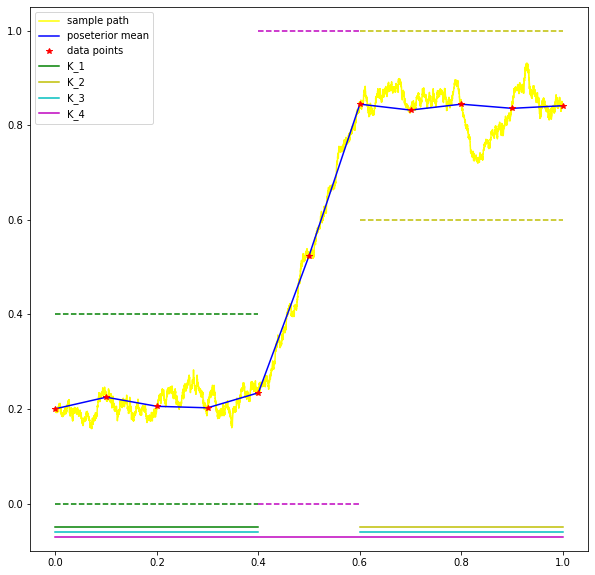}
    \caption{$\sK$ is a lattice of attracting blocks for the sample path shown because the path remains between the indicated (dashed) bounds; these bounds define $\cF_\sK^\text{top}$.}
    \label{Fig:ex2_main}
\end{subfigure}
\caption{}
\label{Fig:ex2_all}
\end{figure}

With this information we are able to calculate $\Prob(\sK \hookrightarrow \sABlock(f))$:
\begin{small}

\[
\prod_{n \in \cI} 
1 - \pi\left(\frac{(\beta(n) - y_n)}{\sigma \sqrt{x_n - x_{n-1}}}, \frac{(\beta(n) - y_{n-1})}{\sigma \sqrt{x_n - x_{n-1}}} ,\frac{(y_n-\alpha(n))}{\sigma \sqrt{x_n - x_{n-1}}}, \frac{(y_{n-1} - \alpha(n))}{\sigma \sqrt{x_n - x_{n-1}}}\right)
\approx 0.9989
\]
\end{small}

Having determined the probability that $\sK$ is a lattice of attractor blocks for $f$ we now provide the associated Morse tiling and Conley indices and discuss what these features indicate about the dynamics of $f$.

The lattice $\sK$ gives the Morse tiling
\[
\cM(\sK) = \setof{M(K_1) = K_1, M(K_2) = K_2, M(K_4) = \text{cl}(K_4\setminus K_3) \mid M(K_4) > M(K_1), M(K_4)>M(K_2)}.
\]
The Conley indices of these Morse tiles are 
\[
\Con_k(M(K_1)) = \Con_k(M(K_2)) = \begin{cases}
\id & \text{if $k=0$} \\
0 & \text{otherwise.}
\end{cases}
\quad\text{and}\quad
\Con_k(M(K_4))  = \begin{cases}
\id & \text{if $k=1$} \\
0 & \text{otherwise}
\end{cases}
\]
This Morse tiling indicates that the dynamical system is bistable, with attractors that contain at least one fixed point in $K_1$ and $K_2$ \cite{szymczak:96,day:frongillo:trevino}.  
Moreover, there is a non-trivial invariant set in $(x_4,x_6)$ with the Conley index of a repelling fixed point.  By Theorem~\ref{Thm:TileBnd}, all of this information is valid for $f$ with probability $0.9989$. 
\end{ex}

Before beginning this final example we make a more general comment.  Given a dynamical system $g:Y \to Y$ and a subset $Y'\subset Y$ such that $g(Y') \subset Y'$, the function $g|_{Y'}:Y' \to Y'$ defines a dynamical system.  Such a restricted system is sometimes of interest when one is concerned only with the local dynamics in the region $Y'$.  These local dynamics can also be understood using Conley theory and the probabilistic techniques that we develop in this paper.  In particular, if one is interested only in understanding the dynamics of $f|_{X'}$, where $X' = \bigcup_{n \in \cI'} [x_{n-1},x_n]$ for $\cI' \subset \setof{1,\cdots,N}$, then we can analyze the dynamics of $X'$ in the same manner as we analyze the dynamics of $X$.  That is, if we let $\cX' $ be a simplicial complex with vertices $\cX'^{(0)} := \setof{x_n}_{n \in \cI'}$ and edges $\cX'^{(1)} := \setof{[x_{n-1},x_{n}]}_{n\in \cI'}$, and for $\cF=\cF_\mu$ or $\cF_\sK$ we define the combinatorial multivalued map $\cF':\cX'\mvmap\cX'$ by $\cF'(\xi):=\cF(\xi)\cap\cX'$, then we are able to give a Morse tiling of $X'$ for $f|_{X'}$ using the same methodology as we did for $X$ and $f$.  

In this final example we exploit this perspective.  We begin by analyzing the dynamics of a Gaussian process defined on the interval $X = [0,1]$. We will see that the lattice of attractor blocks $\sK$ identified for the mean $\mu$ that we define has a fairly low probability of being a lattice of attractor blocks for $f$.  However, we will then note that if we restrict our view to a more local perspective, and instead analyze the same system on the interval $X' = [0.2,0.8]$, we have a reasonably high probability that $f|_{X'}$ contains a periodic orbit.  Hopefully this example demonstrates how an individual who is interested only in certain local information--in this case, a periodic orbit--may increase the probability of seeing the dynamics of interest by restricting their view to a smaller region of phase space.  

\begin{ex}\label{ex: per_orb}

Consider the set 
\begin{align*}
\cT &= \setof{(x_n,y_n) \mid x_n = n/10}_{n=0}^{10}\\
&= \{(0,.03),(.1,.07),(.2,.6853),(.3,.6999),(.4,.6884),\\
&\, (.5,.501),(.6,.255),(.7,.3185),(.8,.2987),(.9,.95),(1,.97)\},
\end{align*}
shown in Figure~\ref{Fig:ex3_all}.  Let $f$ be Brownian motion on $X = [0,1]$ with variance parameter $\sigma^2 = 1/16$ conditioned to interpolate $\cT$.  The mean function $\mu$ of $f$ is the piecewise linear function indicated in blue.  The lattice $\sK = \setof{K_0 = \emptyset, K_1 = [0.2,0.4] \cup [0.6,0.8], K_2 = [0,2,0.8], K_3 = X}$ is a lattice of attractor blocks for $\mu$; we determine this information because $\sK \hookrightarrow \sInvset^+(\cF_\mu^\text{top})$, where $\cF_\mu^\text{top}$ is shown in Figure~\ref{Fig:ex3_F}.  We observe that $K_1$ contains a periodic orbit for $\mu$ and $K_2 \setminus K_1$ contains a repelling fixed point.  Our goal in this example is to show that there is a relatively high probability that $f$ has a periodic orbit and an isolated invariant with the index of an unstable fixed points in these same regions.  

For each of the connected intervals $K_q$, the maps $\alpha_{K_q}$ and $\beta_{K_q}$ are constant maps whose images are, respectively, the left and right endpoints of $K_q$.  However, for $K_1$ the situation is somewhat more complicated.  Note that $\mu([0.2,0.4]) \subset (0.6,0.8)$ and $\mu([0.6,0.8]) \subset (0.2,0.4)$.  Thus 
\[
\alpha_{K_1}(n) = \begin{cases}
0.6, & n = 3,4\\
0.2, & n = 7,8
\end{cases}
\]
and 
\[
\beta_{K_1}(n) = \begin{cases}
0.8, & n = 3,4\\
0.4, & n = 7,8
\end{cases}.
\]
Therefore we have that 
\[
\alpha(n) = \begin{cases}
0, & n = 1,2,9,10\\
0.6, & n = 3,4\\
0.2, & n = 5,6,7,8
\end{cases}
\]
and 
\[
\beta(n) = \begin{cases}
1, & n = 1,2,9,10\\
0.8, & n = 3,4,5,6\\
0.4, & n = 7,8
\end{cases}.
\]
With these maps defined we are able to use Theorem~\ref{Thm:LattBound} to compute
\[
\Prob(\sK \hookrightarrow \sABlock(f)) \approx 0.1199.
\]

This probability is relatively low, but we observe that a sample path is most likely to leave the necessary bounds on $K_4\setminus K_2$; for instance, $\Prob(S_{10}(\alpha(10),\beta(10)) = 0.3812$.  Because we are not concerned with the behavior of $f$ in this region we restrict our view to the domain $X':= K_2$.  For this domain, the lattice $\sK'= \setof{K_0 = \emptyset, K_1, K_2}$ is a full lattice, bounded from above by the domain itself.  Now Theorem~\ref{Thm:LattBound} implies that 
\[
\Prob(\sK' \hookrightarrow \sABlock(f|_{X'})) \approx 0.6283.
\]
\begin{figure}[H]

\centering

\begin{subfigure}{0.48\textwidth}
    \includegraphics[width=\textwidth]{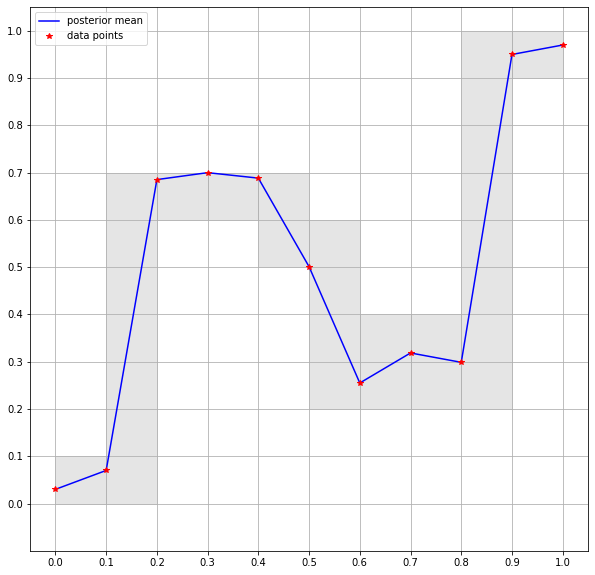}
    \caption{The combinatorial multivalued map $\cF_\mu^\text{top}$.}
    \label{Fig:ex3_F}
\end{subfigure}
\hfill 
\begin{subfigure}{0.48\textwidth}
    \includegraphics[width=\textwidth]{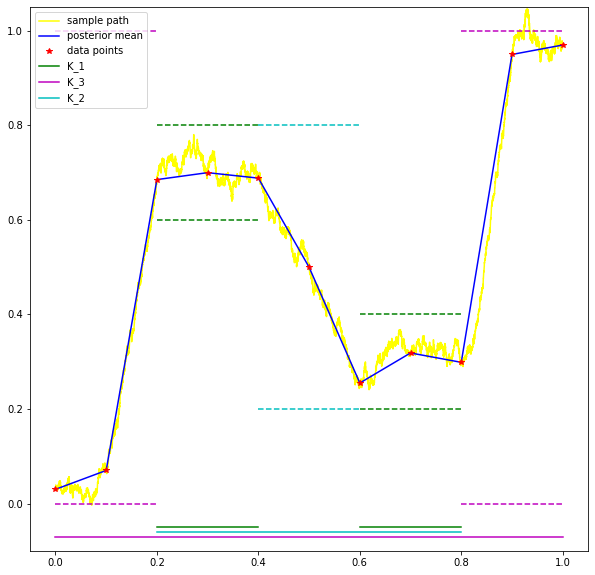}
    \caption{The dashed lines define $\cF_\sK^\text{top}$.  Note that $\cF_\sK^\text{top}$ is not an outer approximation for the sample path shown, but $\cF_{\sK'}^\text{top}$ is an outer approximation of the sample path restricted to $K_2= [0.2,0.8]$.}
    \label{Fig:ex3_main}
\end{subfigure}
\caption{Note that $\sK$ is not a lattice of attracting blocks for the sample path shown because the path leaves the required bounds.  However, $K_1$ and $K_2$ are both attractor blocks for the sample path.  We conclude that the sample path contains a periodic orbit in $K_1$ and an isolated invariant set with the Conley index of an unstable fixed point; the same is true of $f$ with probability $0.6283$.  }
\label{Fig:ex3_all}
\end{figure}

The Morse tiling associated with $\sK'$ is 
\[\cM(\sK') = \setof{M(K_1) = K_1, M(K_2) = \text{cl}(K_2 \setminus K_1) \mid M(K_2) > M(K_1)};
\]
by Theorem~\ref{Thm:TileBnd}, this $\cM(\sK')$ is a Morse tiling for $f$ with probability $0.6283$.  Using $\cF_{\sK'}^\text{top}$, we are able to compute the Conley indices 
\[
\Con_k(M(K_1))  = \begin{cases}
\begin{bmatrix}
0 & 1 \\
1 & 0
\end{bmatrix} & \text{if $k=0$} \\
0 & \text{otherwise.}
\end{cases}
\quad\text{and}\quad
\Con_k(M(K_2))  = \begin{cases}
-\id & \text{if $k=1$} \\
0 & \text{otherwise}
\end{cases}.
\]
These indices imply the existence of a periodic orbit for $f$ in the interior of $K_1$ whenever $\cF_{\sK'}^\text{top}$ is an outer approximation of $f|_{K_2}$ \cite{szymczak:96,day:frongillo:trevino}; thus we conclude that $f$ has a periodic orbit with at least probability $0.6283$.  Therefore, while we cannot make particularly strong claims about the global dynamics of $f$ all of $X$, we are able to identify interesting features of the dynamics on $X' = [0.2,0.8]$ with a reasonable probability.  

\end{ex}


\section{Future Directions and Comparisons to Other Approaches}\label{Sec: concl}

We begin this section by comparing our results to those found in \cite{bogdan}.  The spirit of both papers is the same.  In each of them, combinatorial Conley index theory and Gaussian processes are combined in order to identify dynamics with rigorous probabilities.  However, there are two key differences between these, reflecting a trade-off between the methods used in each case.  

The first difference is that the results of \cite{bogdan} apply to a larger set of Gaussian processes.  The dynamics of any Gaussian process with a covariance kernel that is at least four times differentiable may be studied using the results of that paper.  This requirement means that the Gaussian process $f$ that is analyzed in this paper cannot be studied using the methods of \cite{bogdan}, but the majority of Gaussian processes used in applications--including any process using the squared exponential covariance kernel--can be analyzed using those methods.  

On the other hand, the results of \cite{bogdan} are asymptotic.  In that paper, the set of points $\cT$ that is used to condition the Gaussian process represents sampled data.  The main theorem establishes a procedure that characterizes dynamics where, given a fixed confidence level $\delta\in (0,1)$, it is proven that for a large enough sample $\cT$ the characterization is accurate with confidence level $\delta$.  However, for any single fixed sample $\cT$, the method cannot rigorously identify the confidence that the characterization is accurate.  By contrast, this work can identify the probability that the provided characterization of dynamics is accurate for a fixed $\cT$ and the Gaussian process $f$ interpolating $\cT$.  

Future work should focus on extending the techniques from this paper to higher dimensional data and more general covariance structures.  In order to replicate the results of this paper for more general Gaussian process requires obtaining some understanding of the events $S_n$ in this more general setting.  One natural approach to addressing this problem would be to use the maxima and minima of a Gaussian process in order to estimate the probability of $S_n$.  That is, for a general Gaussian process $G$ with parameter space $\R$, if we let $U_n(\beta) = \setof{\sup_{x\in[x_i,x_j]} G(x) \leq \beta}$ and $L_n(\beta) = \setof{\inf_{x\in[x_i,x_j]} G(x) \geq \alpha}$, then $\Prob(S_n(\alpha,\beta)) \geq 1 - \Prob(U_n(\beta)) - \Prob(L_n(\alpha))$; therefore if we can know (or bound) these probabilities we should be able to obtain lower bounds on the probability that a lattice of attractor blocks identified for the mean of the Gaussian process is also a lattice of attractor blocks for the true map.  One possible route to obtain such bounds would be through Rice's formula \cite{adtay}.

Extending the results to higher dimensional data sets again requires obtaining (or bounding) probabilities like $S_n$.  As mentioned earlier, a closed set $K$ is an attractor block for a map $g$ if $g(K) \subset \Int(K)$; this characterization does not depend on the dimension.  We should be able to give a bound on the probability of such events in any finite dimension $d$ if we can bound the probability that hypercubes map into other hypercubes.  That is, we would like to know or estimate $$\Prob(f(\prod_{k=1}^d[a_k,b_k]) \subset \prod_{k=1}^d[\alpha_k,\beta_k]).$$  While this cubical approach will not allow us to perfectly represent any attractor block in higher dimensions (which have no geometric constraints in general) if we can obtain formulas like this then we will likely be able to extend the results of this paper to higher dimensions with a reasonable level of generality.

Finally, we note that one obvious extension of this work is to develop a statistical procedure which allows us to analyze data sets of the same form as $\cT = \setof{(x_n,y_n)}_{n=0}^N$.  While we worked in a purely probabilistic setting here, we remarked in the introduction that such analysis is our ultimate motivation, and thus future work should attempt to make this generalization. 

\section{Acknowledgements}
The authors would like to thank Harry van Zanten, Ying Hung, and Kasper Larsen; our conversations on Gaussian processes were very valuable in crafting this paper.  

C.T. was partially supported by HDR TRIPODS award CCF-1934924.  K.M. was partially supported by the National Science Foundation under awards DMS-1839294 and HDR TRIPODS award CCF-1934924, DARPA contract HR0011-16-2-0033, and NIH 5R01GM126555-01. K.M. was also supported by a grant from the Simons Foundation.

\bibliography{bib_all.bib} 

\begin{thebibliography}{10}

\bibitem{adtay}
R.~J. Adler and J.~E. Taylor.
\newblock {\em Random fields and geometry}.
\newblock Springer Monographs in Mathematics. Springer, New York, 2007.

\bibitem{bogdan}
B.~Batko, M.~Gameiro, Y.~Hung, W.~Kalies, K.~Mischaikow, and E.~Vieira.
\newblock Identifying nonlinear dynamics with high confidence from sparse data,
  2022.

\bibitem{davey:priestley}
B.~Davey and H.~Priestley.
\newblock Introduction to {L}attices and {O}rder.
\newblock {\em Cambridge University Press}, pages xii+298, 2002.

\bibitem{day:frongillo:trevino}
S.~Day, R.~Frongillo, and R.~Trevi\~{n}o.
\newblock Algorithms for rigorous entropy bounds and symbolic dynamics.
\newblock {\em SIAM J. Appl. Dyn. Syst.}, 7(4):1477--1506, 2008.

\bibitem{doob}
J.~L. Doob.
\newblock Heuristic approach to the {K}olmogorov-{S}mirnov theorems.
\newblock {\em Ann. Math. Statistics}, 20:393--403, 1949.

\bibitem{conj_prob}
M.~Foreman, D.~J. Rudolph, and B.~Weiss.
\newblock The conjugacy problem in ergodic theory.
\newblock {\em Ann. of Math. (2)}, 173(3):1529--1586, 2011.

\bibitem{CMGDB}
M.~Gameiro and S.~Harker.
\newblock {CMGDB}: {C}onley {M}orse {G}raph {D}atabase.
\newblock \url{https://github.com/marciogameiro/CMGDB}, 2022.

\bibitem{gramacy}
R.~B. Gramacy.
\newblock {\em Surrogates---{G}aussian process modeling, design, and
  optimization for the applied sciences}.
\newblock Chapman \& Hall/CRC Texts in Statistical Science Series. CRC Press,
  Boca Raton, FL, [2020] \copyright 2020.

\bibitem{KMV0}
W.~D. Kalies, K.~Mischaikow, and R.~C. A.~M. VanderVorst.
\newblock An algorithmic approach to chain recurrence.
\newblock {\em Found. Comput. Math.}, 5(4):409--449, 2005.

\bibitem{lattice1}
W.~D. Kalies, K.~Mischaikow, and R.~C. A.~M. Vandervorst.
\newblock Lattice structures for attractors {I}.
\newblock {\em J. Comput. Dyn.}, 1(2):307--338, 2014.

\bibitem{lattice2}
W.~D. Kalies, K.~Mischaikow, and R.~C. A.~M. Vandervorst.
\newblock Lattice structures for attractors {II}.
\newblock {\em Found. Comput. Math.}, 16(5):1151--1191, 2016.

\bibitem{lattice3}
W.~D. Kalies, K.~Mischaikow, and R.~C. A.~M. Vandervorst.
\newblock Lattice structures for attractors iii, 2019.

\bibitem{mischmroz}
K.~Mischaikow and M.~Mrozek.
\newblock Conley index.
\newblock In {\em Handbook of dynamical systems, {V}ol. 2}, pages 393--460.
  North-Holland, Amsterdam, 2002.

\bibitem{mischaikow:weibel:22}
K.~Mischaikow and C.~Weibel.
\newblock Conley index and shift equivalence calculations, 2022.

\bibitem{szymczak:96}
A.~Szymczak.
\newblock The {C}onley index and symbolic dynamics.
\newblock {\em Topology}, 35(2):287--299, 1996.

\end{thebibliography}
\bibliographystyle{abbrv}

\end{document}